\documentclass[a4paper,12pt]{amsart}

\usepackage{amssymb,amsbsy,amsmath,amsfonts,amssymb,amscd}
\usepackage{latexsym}
\usepackage{amsxtra}
\usepackage{amscd}
\usepackage{graphics}
\usepackage{epic}
\usepackage{color}
\input xy
\xyoption{all}

\usepackage{ mathrsfs, amsfonts}

\newcommand\sF{{\mathcal F}}

\newcommand\sL{{\mathcal L}}

\newcommand\sB{{\mathcal B}}

\newcommand\sH{{\mathcal H}}
\newcommand\LL{{\mathbb L}}

\newcommand\om{\omega}

\newcommand\Ga{\Gamma}

\newcommand\e{\epsilon}

\newcommand{\CC}{\ensuremath{\mathbb{C}}}

\newcommand{\ZZ}{\ensuremath{\mathbb{Z}}}
\newcommand{\QQ}{\ensuremath{\mathbb{Q}}}

\newcommand{\sS}{\ensuremath{\mathcal{S}}}

\newcommand{\NN}{\ensuremath{\mathbb{N}}}
\newcommand{\hol}{\ensuremath{\mathcal{O}}}

\newcommand{\HH}{\ensuremath{\mathbb{H}}}
\newcommand{\PP}{\ensuremath{\mathbb{P}}}

\newcommand{\HHH}{\ensuremath{\mathcal{H}}}

\newcommand{\ra}{\ensuremath{\rightarrow}}

\def\eea{\end{eqnarray*}}
\def\bea{\begin{eqnarray*}}

\newcommand\dual{\mathrel{\raise3pt\hbox{$\underline{\mathrm{\thinspace d
\thinspace}}$}}}
\newcommand\qe{\ifhmode\unskip\nobreak\fi\quad $\Box$}       

\def\BOX{\hfill\lower.5\baselineskip\hbox{$\Box$}}

\newtheorem{theo}[equation]{Theorem}
\newtheorem{remarkk}[equation]{Remark}
\newenvironment{rem}{\begin{remarkk}\rm}{\end{remarkk}}

\newtheorem{defin}[equation]{Definition}
\newenvironment{definition}{\begin{defin}\rm}{\end{defin}}

\newtheorem{prop}[equation]{Proposition}
\newtheorem{cor}[equation]{Corollary}
\newtheorem{lemma}[equation]{Lemma}
\newtheorem{example}[equation]{Example}

\newtheorem{question}[equation]{Question}

\newcommand{\sR}{\ensuremath{\mathcal{R}}}

\newcommand{\GL}{{\rm GL}}

\newcommand{\Gal}{{\rm Gal}}

\newcommand{\im}{{\rm Im}}

\newcommand{\al}{{\alpha}}

\begin{document}

\title[VBAC coming from   VHS ]{ Vector bundles on curves  coming from Variation of Hodge Structures }
\author{ Fabrizio Catanese - Michael  Dettweiler}
\address {Lehrstuhl Mathematik VIII -Lehrstuhl Mathematik IV\\
Mathematisches Institut der Universit\"at Bayreuth\\
NW II,  Universit\"atsstr. 30\\
95447 Bayreuth}
\email{Fabrizio.Catanese@uni-bayreuth.de}
\email{Michael.Dettweiler@uni-bayreuth.de}

\thanks{
The present work took place in the realm of the ERC advanced grant TADMICAMT,
 the second named author  was supported by the DFG grant DE 1442/4-1}

\date{\today}

\maketitle

\begin{abstract}
  Fujita's  second theorem
 for K\"ahler fibre spaces over a curve asserts  that the direct image $V$ of the relative
dualizing sheaf splits as the direct sum $ V = A \oplus Q$, where $A$ is ample and $Q$ is unitary flat.
We focus on our negative answer (\cite{cd}) to a question by Fujita: is $V$  semiample?

 We  give here an  infinite series of  counterexamples using hypergeometric integrals and
  we give a simple argument to show that  the monodromy representation is infinite.
Our counterexamples are surfaces of general type with positive index,  explicitly given 
  as abelian coverings with group $(\ZZ/n)^2$ of a Del Pezzo surface of degree 5 (branched on a union of lines
forming a bianticanonical divisor), 
and endowed with a   semistable fibration with only $3$ singular fibres.

The simplest such surfaces are the three ball quotients  considered in \cite{bc},
fibred over a curve of genus $2$, and with fibres of genus $4$.

These examples are a larger class than the ones corresponding to
Shimura curves in the moduli space of Abelian varieties.

\end{abstract}

\tableofcontents

\section*{Introduction}
In this paper we first begin  recalling  previous results (\cite{fuj1}, \cite{fuj2}, \cite{cd}, \cite{cd2}) concerning  Fujita's first and second theorem
 for K\"ahler fibre spaces over a curve, asserting  that the direct image $V$ of the relative
dualizing sheaf splits as the direct sum $ V = A \oplus Q$, where $A$ is ample and $Q$ is unitary flat. Then
we focus on our negative answer (\cite{cd}, \cite{cd2}) to a question posed by Fujita 30 years ago: $V$ does not need to be semiample.

 We  show here that the   two examples of  (\cite{cd}) fits into an  infinite series of  counterexamples, again based on the use of hypergeometric integrals \`a la 
 Deligne-Mostow, for each  positive number $n$   and  each  way to write $n$ as a sum of four positive  integers, and yielding a family of cyclic coverings of the line
 parametrized by $\PP^1$.
 
Following Beukers and Heckman we  can show that   the monodromy group of $Q$  is infinite
without  resorting to the classification by H.A.  Schwarz.

Under some mild restrictions on $n$ and the four integers (for $n$ the restriction  boils down to the fact that $n$ should be  coprime to $6$), we give a very simple explicit description 
of fibred surfaces $f : S \ra B$ which are obtained from the above family via  a cyclic $\ZZ/n$-base change  $B \ra \PP^1$, and 
which have the following remarkable properties:

\begin{enumerate}
\item
The Albanese map $\al : S \ra Alb (S)$ has as image a curve of genus $b \geq 2$, and coincides with the fibration  $f : S \ra B$;
\item
all the fibres of $\al$ are smooth, except three singular fibres which are constituted of two smooth curves of genus $b$
meeting transversally in one point
\item
the surfaces $S$ have all positive index, indeed $K^2_S > 2.5 e(S)$
\item
 the direct image $V = f_* (\omega)$ of the relative
dualizing sheaf splits as the direct sum $ V = A \oplus Q$, where $A$ is ample and $Q$ is unitary flat, 
and  $Q$ corresponds to an infinite monodromy representation of $\pi_1(B)$: hence $V$ is not semiample (since, by the results of  \cite{cd}, a unitary flat
bundle  is semi-ample if and only if the monodromy representation is finite).
\end{enumerate}

In the previous two examples (\cite{cd}, \cite{cd2}) we had $n=7$, but we used another method to produce a  base change yielding a semistable 
fibration; as a consequence the degree of the base change that   we needed was  much larger  than $7$ ($42$ in the easier case), 
and  the semistable fibrations were not described with  full details. 
The description we give here was  motivated by a question by Fujino, who asked whether we could give a completely  explicit 
example of a semistable fibration satisfying property (4).

To underline the simplicity of the present geometric construction, let us observe that the  simplest  surfaces in our series
correspond to writing 5 = 2 + 1 + 1 + 1, and are therefore
 surfaces $S$ fibred over a curve of genus $2$ , and with fibres of genus $4$.
It turns out that these surfaces are among  the  ball quotients which were previously considered in \cite{bc}.

The following is our main result.

  \begin{theo}\label{surfaces}
  There exists an infinite series of surfaces with ample canonical bundle,  whose Albanese map is a  fibration $ f : S \ra B$ onto a curve $B$ of genus $b= \frac{1}{2}(n-1)$, and with fibres of genus $g = 2b= n-1$,  where $n$ is any integer relatively prime with $6$.
  
  These Albanese  fibrations  yield negative answers to Fujita's question about the semiampleness of $V : = f_* \om_{S|B}$, since
  here $V : = f_* \om_{S|B}$
splits as a direct sum 
$ V = A  \oplus Q$, where $A$ is an ample    vector bundle, and $Q$ is a  unitary  flat  bundle with
infinite monodromy group.

The fibration $f$ is semistable: indeed all the fibres are smooth, with the exception of three fibres which are the union of two smooth curves of genus $b$
which meet transversally in one point.

For $n=5$ we get   three surfaces which are rigid, and are quotient of the unit ball in $\CC^2$ by a torsion free cocompact lattice $\Ga$.
The rank of $A$, respectively $Q$, is in this case equal to $2$.
 \end{theo}

Finally we end surveying quite briefly  relations with existing literature concerning
Shimura curves in the moduli space of Abelian varieties: this is work of several people, but especially the work of Moonen (\cite{moonen})
is related to our easiest examples.

\section{Fujita's theorems and questions on Vector Bundles on curves arising from Variation of Hodge Structures}

An important progress in classification theory was stimulated by  a theorem of Fujita, who showed (\cite{fuj1}) that if $X$ is a compact K\"ahler
manifold and $ f : X \ra B$ is a fibration onto a projective curve $B$ (i.e., $f$ has connected fibres), then the direct image sheaf
$$  V : = f_* \om_{X|B} = f_* ( \hol_X (K_X - f^* K_B))$$ is a  nef   vector bundle on $B$, where `nef'  means that each quotient bundle $Q$
of $V$ has degree $\deg (Q) \geq 0$; sometimes the word `nef' is  replaced by the word  `numerically semipositive'.

In the note \cite{fuj2} Fujita announced the following quite  stronger result:

\begin{theo}{\bf (Fujita, \cite{fuj2})}\label{fuj2}

Let $f : X \ra B $ be  a fibration of a compact K\"ahler manifold $X$ over a projective curve $B$, and consider 
the direct image  sheaf $$ V : = f_* \om_{X|B} = f_* ( \hol_X (K_X - f^* K_B)).$$
Then $V$ splits as a direct sum  $ V  = A \oplus Q$, where $A$ is an ample vector bundle and $Q$ is a unitary  flat bundle.
\end{theo}

Fujita sketched the proof, but referred to a forthcoming article concerning the positivity  of the so-called local exponents (this article was never written, see \cite{Barja}).

 Soon afterwards, using Griffihts' results on Variation of Hodge Structures,
since the fibre of $V: = f_* \om_{X|B}$ over a point $b \in B$ such that $X_b : = f^{-1} (b)$ is smooth is  the vector space 
$  V_b   = H^0( X_b, \Omega^{n-1}_{X_b}),$ 
   Kawamata (\cite{kaw0} \cite{kaw1}) improved on Fujita's result, solving a long standing problem
   and proving  the subadditivity of Kodaira dimension for
 such fibrations, $$ Kod (X) \geq Kod(B) + Kod (F),$$ (here $F$ is a general fibre). Kawamata did this by  showing the semipositivity also for 
the direct image of  higher powers of the relative dualizing sheaf 
$$W_m : =  f_* ( \om_{X|B}^{\otimes m}) = f_* ( \hol_X (m(K_X - f^* K_B))).$$

  Kawamata also extended his result to the case where the dimension of the base variety $B$  is $> 1$  in \cite{kaw0},
giving later a simpler proof of semipositivity  in \cite{kaw2}. There has been a lot of literature on the subject ever since,
see the references we cited (see \cite{e-v} for the ampleness of $W_m$ when $m \geq 2$ and when  the fibration is not birationally  isotrivial,
see also  \cite{ff14} and \cite{ffs14}). 
Kawamata introduced  a simple lemma, concerning
the degree of line bundles on  a curve whose metric grows at most  logarithmically around a finite number of singular points, which played a crucial role for the proof.

\medskip

The missing details concerning the proof of the second theorem of Fujita, using Kawamata's lemma and some crucial estimates
given by Zucker (\cite{zucker}) for the growth of the norm of sections of the $L^2$-extension of Hodge bundles, were provided in \cite{cd}, 
where also a negative answer was given to the following question posed by Fujita in 1982 (Problem 5, page 600 of \cite{katata},
Proceedings of the 1982 Taniguchi  Conference).

To understand this question it  is not only  important to have in mind Fujita's second theorem, but it is also very convenient to recall the following 
 classical definition used by Fujita in \cite{fuj1}, \cite{fuj2}.

 Let $V$ be a holomorphic vector bundle  over a projective curve $B$.

\begin{definition} 
 Let $p :  \PP : = Proj (V) = \PP (V^{\vee}) \ra B$ be the associated projective bundle, and  let $H$ be a  hyperplane divisor (s.t.
 $p_* (\hol_{\PP} (H)) = V$).

Then $V$ is said to be:

(NP)  numerically semi-positive if and only if every quotient bundle $Q$ of $V$ has  degree $deg(Q) \geq 0$,

(NEF) nef if and only if $H$ is nef   on $\PP$,

(A) ample  if and only if $H$ is ample  on $\PP$ 

(SA) semi-ample  if and only $H$ is semi-ample  on $\PP$ (there is a positive multiple  $mH$ such that the linear system $|mH|$ is base point free).
\end{definition}
\begin{rem}

Recall that  (A) $\Rightarrow$ (SA) $\Rightarrow$ (NEF) $\Leftrightarrow$(NP), 
the last  follows from the following result  due to Hartshorne.

\end{rem}
\begin{prop}\label{ample}
A vector bundle $V$ on a curve   is  nef if and only it is numerically semi-positive, i.e.,  if and only if every quotient bundle $Q$ of $V$ has  degree $deg(Q) \geq 0$,
and $V$ is ample  if and only if every quotient bundle $Q$ of $V$ has  degree $deg(Q) > 0$.

\end{prop}

Moreover, we have also:

\begin{definition}
A flat holomorphic vector bundle  on a complex manifold $M$ is  a holomorphic vector bundle $\sH : = \hol_M \otimes_{\CC} \HH$,
where $\HH$ is a local system of complex vector spaces associated to a representation $\rho : \pi_1(M) \ra GL  (r, \CC)$,
$$  \HH : =( \tilde{M} \times \CC^r )/  \pi_1(M) ,$$
$ \tilde{M} $ being the universal cover of $M$ (so that $ M =  \tilde{M} /  \pi_1(M)$).

 We say that $\sH$ is  unitary flat if it is associated to a  representation $\rho : \pi_1(M) \ra U (r, \CC)$.

\end{definition}
\bigskip
 
 \begin{question} {\bf (Fujita)}
 Is the direct image  $V : = f_* \om_{X|B}$ semi-ample ?
 
 \end{question}

In \cite{cd} we established    a  technical result  which clarifies how Fujita's question is very closely related to Fujita's  II theorem
\begin{theo}\label{semiample}
Let $\sH$ be a unitary flat vector bundle on a projective manifold $M$, associated to a representation $\rho : \pi_1(M) \ra U (r, \CC)$.
Then $\sH$ is nef and moreover $\sH$ is semi-ample if and only if $ Im (\rho )$ is finite.
\end{theo}

Hence in  our particular case, where $ V  = A \oplus Q$ with $A$ ample and $Q$ unitary  flat, the semiampleness of $V$  simply means that the flat bundle has finite monodromy
(this is another way of wording the fact that  the representation of the fundamental group 
 $ \rho : \pi_1(B) \ra  U (r, \CC)$ associated to the flat unitary
 rank-r bundle $Q$  has finite image).

The main new result in our joint work \cite{cd} was to provide  a negative answer to Fujita's question in general:
\begin{theo}\label{maintheorem}
There exist  surfaces $X$ of general type endowed with  a fibration $ f : X \ra B$ onto a curve $B$ of genus $\geq 3$, and with fibres of genus $6$,  such that $V : = f_* \om_{X|B}$ 
splits as a direct sum $ V = A  \oplus Q_1 \oplus Q_2$, where $A$ is an ample   rank-2  vector bundle, and the flat unitary rank-2 summands $Q_1, Q_2$
have infinite monodromy group (i.e., the image of $\rho_j$ is infinite). In particular, $V$ is not semi-ample.

 \end{theo}
 
 Recall however that in special cases  one can conclude that  $V$ is semiample.

 \begin{cor}

Let $f : X \ra B $ be  a fibration of a compact K\"ahler manifold $X$ over a projective curve $B$.
Then  $ V : = f_* \om_{X|B} $
is  a direct sum  $ V  = A \bigoplus (\oplus_{i=1}^h Q_i)$, 
with $A$   ample  and each $Q_i$  unitary  flat 
without any nontrivial degree zero quotient. 
Moreover,

(I) if $Q_i$ has rank equal to 1, then it is a torsion bundle ($\exists \  m$ such that $Q_i^{\otimes m}$
is trivial) (Deligne)

(II) if the  curve $B$ has genus 1, then rank  $(Q_i) = 1 , \ \forall i.$ 

(III)  In particular,  if  $B$ has genus  at most  1, then $V$ is semi-ample.

\end{cor}

\begin{proof}
{The idea of the proof is as follows:}

(I) was proven by Deligne  (and by Simpson using the theorem of Gelfond-Schneider), while

(II) Follows since  $\pi_1(B)$ is abelian, if $B$ has genus $1$: 
hence every representation splits as a direct sum of 1-dimensional ones.

\end{proof}

  In our construction for theorem  \ref{maintheorem}, we started from  hypergeometric integrals associated to a cyclic group of order 7,
 and we derived the non finiteness of the monodromy as a consequence of the classification
 due to  Schwarz (\cite{Schwarz}).
 
 However, in order to provide a semistable fibration, we first resolved the singularities of the resulting surface fibres over $\PP^1$,
 then applied blow ups in order to achieve that the reduced divisors associated to the fibres would be normal crossing divisors,
 and then applied the general method in order to construct a semistable base change.
 
 The final result was that these examples had a base of much larger genus, and the description given was not fully detailed.
 The novelty of this paper, answering a question by Osamu Fujino, is to provide an explicit semistable fibration without 
 having to take a base change where the genus of the base curve $B$ becomes too large. This will be discussed in the fourth section,
 where we shall also give a simpler proof.

 An interesting observation, concerning the crucial difference of the roles played by unitary flat bundles versus  flat bundles 
in our context,   is given  by the following result.
While a unitary flat bundle is nef, the same does not hold for a flat bundle. This is no surprise,
as communicated to the first author by  Janos Koll\'ar, in view of  the following old theorem  of Andr\'e
  Weil \cite{weil}, reproven by Atiyah in  \cite{atiyah}.

 \begin{theo}{\bf (Weil-Atiyah)}
A vector bundle $V$ over a projective curve is (isomorphic to) a flat holomorphic bundle if and only if,    in
its unique decomposition as a direct sum $V = \oplus_i V_i$ of indecomposable bundles, each of the summands $V_i$ has degree zero.
\end{theo}

In our situation we proved (again in \cite{cd}): 
  \begin{theo}\label{Kodaira}
Let  $ f : X \ra B$  be a Kodaira fibration, i.e., $X$ is a surface and all the fibres of $f$ are smooth curves not all isomorphic to each other.
Then the direct image sheaf  $V : = f_* \om_{X|B}$ has strictly positive degree hence $\HHH : = R^1 f_* (\CC) \otimes \hol_B$ is a flat bundle which is not nef (i.e., not numerically semipositive).
 \end{theo}

\subsection{Semistable reduction}

Assume now that $ f : X \ra B$ is a fibration of a compact K\"ahler manifold $X$ over a projective curve $B$, and consider 
the invertible sheaf $ \om : = \om_{X|B} = \hol_X (K_X - f^* K_B).$ 

By Hironaka's theorem there is a sequence of blow ups with smooth centres $ \pi: \hat{X} \ra X$ such that $$ \hat{f} := f \circ \pi: \hat{X} \ra B$$
has the property that all singular  fibres $F$ are such that $ F = \sum_i m_i F_i$, and $F_{red} =  \sum_i  F_i$ is a normal crossing divisor.

Since $\pi_* \hol_{\hat{X} } (K_{\hat{X} }) = \hol_X (K_X)$ we obtain 

$$ \hat{f} _* \om_{\hat{X}|B } = \hat{f} _* \hol_{\hat{X} } (K_{\hat{X} } - \hat{f} ^* K_B ) = f_* \hol_X (K_X - f^* K_B) = f_*  \om_{X|B} .$$

Therefore one can assume wlog that all the   fibres of $f$ have reduction which is a  normal crossing divisor, and the well known semistable reduction theorem, whose statement is
here reproduced, shows that one can reduce to the case where the fibration is semistable, i.e., all fibres are reduced and yield normal crossing divisors.

\begin{theo} {\bf (Semistable reduction theorem, \cite{toroidalEmb})}
There exists a cyclic Galois covering of $B$, $B' \ra B = B' /G$, such that the normalization $X''$ of the fibre product $ B' \times_B X$
admits a resolution $ X' \ra X''$ such that the resulting fibration $f' : X' \ra B'$ has the property that all the fibres  are reduced and 
normal crossing divisors.
\end{theo}

\begin{equation*}
\xymatrix{
X' \ar[d]^{f'}\ar[r]^{v'} & X'' \ar[d]\ar[r]^{v''}  & X\ar[d]^f\\
B' \ar[r]^{Id} &B'  \ar[r]^u & B ,}
\end{equation*}

The following proposition was used in \cite{cd} while  reducing the proof of Fujita's second theorem to the semistable case.

\begin{prop}\label{sstablered}
The sheaf $V' : =   f'_* \om_{X'|B'}$ is a subsheaf of the sheaf $ u^* (V)$, where $V : =   f_* \om_{X|B} $, and the cokernel 
$ u^* (V) / V'$ is concentrated on the set of points corresponding to singular fibres of $f$.

\end{prop}

The proposition  shows indeed that, when the fibration is not semistable,
then certain unitary flat summands on $B'$ may yield ample summands on $B$; and the precise calculation given in its proof
helps to decide exactly when this happens.

\section{Fujita's second  theorem}

The tools used for the proof of Fujita's second  theorem involve 
 differential geometric notions of positivity, which  we now recall.

\begin{definition}
Let $(E,h)$ be a  Hermitian vector bundle on a complex manifold $M$. Take the canonical Chern connection 
associated to the Hermitian metric $h$, and denote by $\Theta(E,h)$ the associated Hermitian curvature,
which gives a Hermitian form on the complex vector bundle bundle $T_M \otimes E$.

Then  one says that $E$ is Nakano positive (resp.: semi-positive) if there exists a Hermitian metric $h$
such  that the Hermitian form associated to $\Theta(E,h)$ is strictly positive definite (resp.: semi-positive definite).

\end{definition}

\begin{rem}
Umemura proved (\cite{um})  that a vector bundle $V$ over a curve $B$ is positive (i.e.,  Griffiths positive, or equivalently Nakano positive) if and only if $V$ is ample.
\end{rem}

One of the principal positivity property can be summarized through the well known slogan: `curvature decreases in subbundles'.

Except that one has to formulate the statement properly as follows: curvature decreases in Hermitian subbundles.
Indeed the example of Kodaira fibrations produces subbundles of a flat bundle (they have zero curvature)
which are positively curved.

We pass now to sketching the ideas used in the proof of Fujita's second theorem.

\begin{theo}{\bf (Fujita, \cite{fuj2})}\label{fuj2}

Let $f : X \ra B $ be  a fibration of a compact K\"ahler manifold $X$ over a projective curve $B$, and consider 
the direct image  sheaf $$ V : = f_* \om_{X|B} = f_* ( \hol_X (K_X - f^* K_B)).$$
Then $V$ splits as a direct sum  $ V  = A \oplus Q$, where $A$ is an ample vector bundle and $Q$ is a unitary flat bundle.

\end{theo}

\subsection{Sketch of proof of Fujita's theorem}\

I)  Thanks to the auxiliary results shown in the previous section,
using the semistable reduction theorem (yielding a base change $B' \ra B$ such that all fibres of the pull-back
$X' \ra B'$ are reduced with normal crossings) and in particular proposition \ref{sstablered}, 
giving a  comparison of  the pull-back of $V$ with the analogously defined $V'$,
it suffices to prove the theorem in the semistable case, i.e., where each fibre is reduced and a normal crossing divisor
 (see proposition 2.9 of \cite{cd} for  details).

II) Idea of the proof in the case of no singular fibres.

$V$ is a holomorphic subbundle of the holomorphic  vector bundle $\sH$  associated to the local
system $\HH_{\ZZ}: = \sR^{m} f_* (\ZZ_X), \ m := \dim(X) - 1$ (i.e.,  $\sH = \HH_{\ZZ} \otimes_{\ZZ} \hol_B$).

The bundle $\sH$ is flat, hence the curvature $\Theta_{\sH} $ associated to the flat connection satisfies 
$\Theta_{\sH} \equiv 0$.

We  view $V$ as a holomorphic subbundle of $\sH$, while  $$V^{\vee} \cong R^m f_* \hol_X, \ m = \dim(X) - 1$$ is a holomorphic quotient bundle of $\sH$.

The curvature formula for subbundles gives ($\sigma$ is the II fundamental form)
$$  \Theta_{V} =  \Theta_{\sH} |_V + \bar{\sigma} \ ^t \sigma = \bar{\sigma} \ ^t \sigma ,$$
and Griffiths  (\cite{griff2}, see also \cite{Griffiths_Topics} and \cite{zucker-remarks}) proves that the curvature of $V^{\vee}$ is semi-negative, 
since its local expression is of the form
$ i h' (z) d \bar{z} \wedge dz $, where $h'(z)$ is a semi-positive definite Hermitian matrix.

In particular we have that  the curvature  $  \Theta_{V} $  of $V$
is semipositive and, 
moreover,  that the curvature  vanishes identically  if and only if the second fundamental form $\sigma$
vanishes identically, i.e., if and only if $V$ is a flat subbundle.

However, by semi-positivity, we get that  the curvature  vanishes identically  if and only its integral, the degree of $V$,
equals zero. Hence $V$ is a flat bundle if and only if  it has degree $0$.

The same result then  holds true, by a similar  reasoning,  for each holomorphic quotient bundle $Q$.

III) The more difficult part of the proof uses some crucial estimates
given by Zucker (using Schmid's asymptotics for Hodge structures) for the growth of the norm of sections of the $L^2$-extension of Hodge bundles, and 
the following  lemma by Kawamata (\cite{kaw1}, see also proposition 3.4, page 11 of  \cite{Peters}).  

\begin{lemma}
Let $L$ be a holomorphic line bundle over a projective curve $B$, and assume that 
$L$ admits a singular metric $h$ which is regular outside of a finite set $S$ and  has at most logarithmic growth at
the points $ p \in S$.

Then the first Chern form $c_1(L,h) := \Theta_h $ is integrable on $B$, and its integral equals $ deg(L)$.
\end{lemma}

The above lemma  shows that in the semistable case singularities are ininfluent, and the argument runs as in the case of no singular fibres.

IV) 
The existence of such a metric follows from the results of Schmid in \cite{schmid} and Zucker in \cite{zucker},
leading to the following lemma.

\begin{lemma}
For each point $s \in B$ there exists a  basis of $V$ given by elements $\sigma_j$ such that their norm 
in the flat metric outside the punctures grows at most logarithmically.

In particular, for each quotient bundle $Q$ of $V$ its determinant admits a metric with growth at most logarithmic at the punctures $ s \in S$,
and the degree of $Q$ is given by the integral of the first Chern form of the singular metric.

\end{lemma}

\section{Cyclic coverings of the projective line branched on $4$ points.}\label{sectioncyclic}

In this section we explain how we obtain explicit examples of fibrations where $V = f_* \omega$ has a flat summand.
Let $\zeta_n:=e^{\frac{2\pi i}{n}}.$
Consider a cyclic covering of  the projective line with group $\ZZ/n$, branched on four points. Hence a curve $C=C_x$
described by an equation
\begin{equation}\label{eq51} z_1^n = y_0^{m_0} y_1^{m_1} (y_1 - y_0 )^{m_2}  (y_1 - x y_0 )^{m_3} , \ \  x \in \CC \setminus \{0,1  \},
\end{equation}
where, of course, ${\rm gcd}(m_0,\ldots,m_3,n)=1.$ 

The above equation describes  a singular curve inside the line bundle over  $\PP^1$ whose sheaf of sections is the sheaf $\hol_{\PP^1}(1)$,
and we denote by $C$ the normalization of this curve. Then $C$   admits a Galois cover  $ \phi : C \ra \PP^1$ with cyclic 
Galois group equal to the group of n-th roots of unity in $\CC$, $$ G = \{ \zeta \in \CC^* | \zeta^n = 1\},$$
acting by scalar multiplication on $z_1$.
The choice of a generator in $G$ yields an isomorphism $ G \cong \ZZ/n$, for instance we have 
$ G=\langle \e\rangle$, where  $\e$ acts as $ z_1 \mapsto \zeta_n z_1.$
 The cover $\phi$ is branched at $\sS=\{s_1=0,s_2=1,s_3=x,s_0=\infty\}$ 
 (where, in projective coordinates $[y_0,y_1]$ one has 
   $0=[1,0],\,1=[1,1],\, x=[1,x],\, \infty=[0,1]$).
   
   We shall make the restrictive assumption that $$ 0 < m_j \leq n-3, \ {\rm  and} \ \  m_0 + m_1 + m_2 + m_3= n.$$ 

 \begin{rem}
We want to point out that the above is indeed a restriction, even if we allow a change of the generator of $G$ taking the residue class of a number $h$ coprime to $n$.
This change of generator has the effect of replacing $m_j$ with the rest  modulo $n$ of $ hm_j$, which we denote by $[hm_j]$.

Now, take the example where   $n=8$ and $m_1 = m_2 = 4, m_3 = 3, m_4 = 5$. Then, however we change the generator
of $G$,  the residue class $[h4]$ shall always be equal to $4$. Hence the  sum $[hm_0 ]+ [hm_1] + [hm_2] + [hm_3 ]>  n= 8$ always,
and indeed the sum   is then always equal to $2n = 16$ (observe in fact that $\Sigma:= ([hm_0 ]+ [hm_1] + [hm_2] + [hm_3 ]) \in \{n, 2n. 3n\}$,
and, changing $m_j$ to $n-m_j$, $\Sigma \mapsto 4 n - \Sigma$). 
\end{rem}  
   
Now, for $j\in \ZZ/n\ZZ$, let $\chi_j:G\to \CC^*, \zeta \mapsto \zeta^j$ and let $ \LL_j$ denote the 
rank-one local system on $\PP^1 \setminus \sS$ whose monodromy matrix $\alpha_s$ at $s_i$ is given by $\zeta_n^{m_i\cdot j}.$

 Let $H^1(C,\CC)_j$ be the subspace of $H^1(C,\CC)$ on which $G$ acts as $\chi_j.$ 
 Then one has an isomorphism 
 $$ H^1(C,\CC)_j=H^1(\PP^1 \setminus \sS,\LL_j)$$ and moreover
 $$ H^{1,0}(\PP^1 \setminus \sS,\LL_j)=H^{1,0}(C)_j,$$ 
 where  $H^{1,0}(C)_j$ is again the part of $H^{1,0}(C)$ on which $G$ acts by the character  $\chi_j$
 (cf.~\cite{DeligneMostow}, ~Section~(2.23)). For $j\neq 0,$ one has  $$\dim(H^1(C,\CC)_j)=\dim(H^1(\PP^1 \setminus \sS,\LL_j))=2.$$ 
For $j\neq 0$ let $\mu_{i,j}=\frac{[m_i\cdot j]}{n} .$  By 
\cite{DeligneMostow}, Equation 2.20.1,
\begin{equation}\label{eq2}
 \dim H^{1,0}(\PP^1 \setminus \sS,\LL_j)=-1+\sum_{i=0}^3\mu_{i,j}\quad \textrm{for} \quad j\in \ZZ/n\ZZ, \, j\neq 0 .\end{equation} 
Hence, under the above assumption \eqref{eq51} we have 
$$ H^1(C,\CC)_{-1}=H^{1,0}(C,\CC)_{-1}\simeq H^{1,0}(\PP^1 \setminus \sS,L_{-1}).$$ 

By \cite{DeligneMostow}, Prop. 2.20,
 the Hermitian form $H_j$ on $H^1(\PP^1\setminus \sS,\LL_j)$ 
  is positive definite on $H^{(1,0)}(P\setminus \sS,\LL_j)$ and negative definite on $H^{(0,1)}(\PP^1 \setminus
\sS,\LL_j).$ Hence the positivity and the negativity  index  of $H_j$ are given by 
\begin{equation}\label{eq52} (-1+\sum_{i=0}^3\mu_{i,j}, 3 - \sum_{i=0}^3\mu_{i,j}).\end{equation}

Varying $x\in \CC\setminus \{0,1\},$ one obtains a family of curves $\pi:\mathcal{C}\to \CC\setminus \{0,1\}$
with fibre $\pi^{-1}(x)=C_x,$ equipped  with a compatible action of $G.$ For each $j\in \ZZ/n\ZZ$ one also obtains a local system 
$\LL'_j$ on 
$$M:=\{(x,y)\in \CC^2 \mid x,y\neq 0,1,\, x\neq y\}$$ which extends $\LL_j$ to $M.$ 
 Let 
$$f:M\to \CC\setminus \{0,1\},\, (x,y)\mapsto x.$$ 
The  higher direct image $\hat{\HH}=R^1\pi_*(\CC)$ decomposes then  
with respect to the $G$-action into $\chi_j$-equivariant parts  
$$ \hat{\HH} =\bigoplus_{j\in \ZZ/n\ZZ} \hat{\HH}_j\quad \textrm{where}\quad \hat{\HH}_j=R^1f_*\mathbb{\LL}_j'.$$

If $j\neq 0,$ then the monodromy representation of $\hat{\HH}_j$  by \eqref{eq52} respects  
a Hermitian form $H_j$  of index 
\begin{equation}\label{eq52} (-1+\sum_{i=0}^3\mu_{i,j}, 3 - \sum_{i=0}^3\mu_{i,j}).\end{equation}

We shall prove in the appendix that this monodromy representation is  irreducible if the $m_i$'s are coprime to $n$ (as we shall assume in the sequel).
This result and the following lemma
shall be used to show the existence of a flat unitary summand with infinite monodromy:  but we shall also give a self-contained 
and more elementary proof of the main theorem which only uses the second Fujita theorem.

The index of $H_j$ is related to finiteness properties of the monodromy of $\hat{\HH}_j$ by the following result
which is a straightforward generalization of \cite{BH}~Thm.~4.8 (cf. \cite{haraoka}). We remark that 
the lemma applies in many other contexts, e.g., for more general rigid local systems or motivic local systems 
whose Hodge numbers can be calculated 
(cf. \cite{Ka96}, \cite{drk}, \cite{ds}).

\begin{lemma}\label{lemmafinite} Choose an embedding of $\bar{\QQ}$ into $\CC.$ 
Let $K$ be either a finite abelian extension of $\QQ$ or
a totally real Galois extension of $\QQ.$ Denote by $\Gal:= \Gal(K/\QQ)$ and let $\hol_K$ denote the 
ring of integers of $K.$   
 Let $\Ga$ be a finitely generated group and let $\rho: \Ga \to \GL_n(\hol_K)$ be an absolutely irreducible 
 representation whose image shall be denoted by  $H: =\im(\rho).$ Suppose that $\rho$ respects a Hermitian form, i.e., there exists 
 a Hermitian matrix $M=(m_{i,j})\in K^{n\times n}$ with  
 $$ \bar{A}^TMA=M\quad \forall A\in H.$$  
 For $\sigma \in \Gal,$ let $M^{\sigma}=(m_{i,j}^\sigma).$ 
  Then $H$ is finite if and only if $M^\sigma$ is a definite Hermitian form for all $\sigma\in \Gal.$ 
\end{lemma}

\begin{proof} If $H$ is finite then $H$ leaves the positive definite unitary form
$$ \bar{v}^TMw:=\sum_{h\in H} \bar{v}^T\bar{h}^Thw\quad v,w\in K^n $$ invariant. By our assumptions on $K,$
any $\sigma\in \Gal$ commutes with  complex conjugation, hence $H^\sigma$ leaves the form 
defined by the matrix $M^\sigma$ invariant. Moreover, $M^\sigma$ is  determined up to a constant, since 
$H^\sigma$ is again irreducible. Since $H^\sigma$ is also finite, the matrix $M^\sigma$ must be definite.

 Let now  the form $M^\sigma$ be definite for any $\sigma \in \Gal.$ By the additive isomorphism 
$\hol_K\simeq \ZZ^d\, (d=|\Gal|),$ 
the representation $\rho$ gives rise to a representation 
$$\tilde{\rho}:\Ga\to \GL_{nd}(\ZZ)$$ such that he trace of $\tilde{\rho}(g)$ coincides with the 
the relative trace of $\rho(g):$ 
$$ {\rm Trace}(\tilde{\rho}(g))={\rm Trace}_{K/\QQ}({\rm Trace}(\rho(g))=\sum_{\sigma\in \Gal} {\rm Trace}(\rho(g))^\sigma.$$ 
Extending the scalars from $\ZZ$ to $\CC$ we obtain from $\tilde{\rho}$ a representation
$$\tilde{\rho}\otimes \CC:\Ga\to \GL_{nd}(\CC).$$
Since any semisimple representation with values in $\CC$ is  determined up to isomorphy by its trace 
by the theorem of Brauer-Nesbitt, there exists  a matrix $g\in \GL_n(\CC)$ such that 
$$ (\tilde{\rho}\otimes\CC)^g=\prod_{\sigma\in \Gal}\rho^\sigma\otimes \CC,$$
where  $\rho^\sigma\otimes \CC$ denotes the extension of scalars of $\rho^\sigma$ from $\hol_K$ to $\CC.$ 
By the definiteness of the forms $M^\sigma\, (\sigma \in \Gal),$
the latter representation takes its values in the product $\prod_{\sigma\in \Gal}U(M^\sigma)$ 
of the compact unitary groups
$U(M^\sigma)$  associated to 
the Hermitian forms $M^\sigma\, (\sigma \in \Gal).$ We conclude that the image of $\tilde{\rho}$ is contained
in the compact and discrete group $$\left(\prod_{\sigma\in \Gal}U(M^\sigma)^{g^{-1}}\right) \cap \GL_{nd}(\ZZ)$$
and is hence finite. Therefore the image of $\rho$ is finite.
\end{proof}

The next result is a straightforward consequence of the above lemma (of course, it may also be derived by 
Schwarz' list of hypergeometric differential equations with finite monodromy \cite{Schwarz}). For simplicity, we restrict ourselves 
to the cases considered in the next section.

\begin{cor}\label{corinfinite} Assume that 
$$  n \in \NN,   n \geq 5, {\rm such \  that} \ GCD (n,6) = 1, \ m_0, m_1, m_2, m_3 \in \NN,  $$
$$ {\rm with} \ 1 \leq   m_j \leq n-1, \  \  m_0 +  m_1 +  m_2 +  m_3 = n, $$
$$ m_k,    (m_i +  m_3 )\in  (\ZZ / n\ZZ)^*, \forall i=0,1,2, k = 0,1,2,3.$$
Then there is $j\in (\ZZ/n\ZZ)^\ast.$ such that the  monodromy of the   local system $\hat{\HH}_j$  is infinite.  
\end{cor}

\begin{proof} Since $n-1$ and $m_0+m_3$ are    invertible in $\ZZ/n\ZZ,$  there is a $j$ such that $j (m_0+m_3) \equiv -1 \ (mod \ n)$.

Define now $m'_i : = [ m_i j]$, so that $m'_0+m'_3= n-1$.

We have 
the obvious inequalities $  2\leq m'_1+m'_2\leq 2n-2.$ Hence 
$$ n+1\leq m'_0+\cdots+m'_3\leq 3n-3$$ and therefore 
$$  m'_0+\cdots+m'_3= 2n.$$ 
Hence the underlying unitary form is indefinite by Formula~\eqref{eq52},
hence the monodromy is infinite by Lemma \ref{lemmafinite} and Prop.~\ref{propapp}.
\end{proof}

\section{Abelian coverings of $\PP^1 \times \PP^1$ yielding surfaces which are counterexamples to Fujita's question}

In this section we shall provide an infinite series of  examples of surfaces fibred over a curve, whose fibres are curves with a symmetry of $ G : = \ZZ/n$
( and with quotient $\PP^1$).

To avoid too many technicalities, we make the following simplifying assumptions, part of which were already mentioned in corollary \ref{corinfinite}:

$$  n \in \NN,   n \geq 5, {\rm such \  that} \ GCD (n,6) = 1, \ m_0, m_1, m_2, m_3 , n_0, n_1, n_2 \in \NN,  $$
$$ {\rm with} \ 1 \leq  n_i, m_j \leq n-1, \  \  m_0 +  m_1 +  m_2 +  m_3 = n, n_0 +  n_1 +  n_2  = n$$
$$ m_j,   n_i,  (m_i +  m_3 )\in  (\ZZ / n)^*, \forall i=0,1,2, j = 0,1,2,3.$$

\begin{lemma}
Such integers $n_i, m_j$ satisfying the above properties exist if and only if $n$ and $6$ are coprime.
\end{lemma}

\begin{proof}
If $n$ is even, then if $m_j$ is a unit in $\ZZ/n$, then $m_j$ is odd, but then $m_i + m_3$ is even,
and cannot be a unit.

If instead $ 3 | n$, then without loss of generality $$m_0 \equiv m_3 \ (mod \ 3),   m_1 \equiv m_2 \equiv -m_3 \ (mod \ 3), $$
but then $m_1 + m_3$  is not a unit in $\ZZ/n$.

Finally, if $GCD (n,6) = 1$, then we can simply choose 
$$  m_0= m_1=  m_2= 1, \ m_3 =  n-3,  n_0= n_1 = 1 , n_2 = n-2.$$

\end{proof}

\begin{definition}
We shall refer to the choice $$  m_0= m_1=  m_2= 1, \ m_3 =  n-3,  n_0= n_1 = 1 , n_2 = n-2$$
as the {\bf standard case}.
\end{definition}

 We consider again the equation 
$$ z_1^n =  y_0 ^{m_0} y_1^{m_1}  (y_1 - y_0 )^{m_2}   (y_1 - x y_0 )^{m_3}  , \ \  x \in \CC \setminus \{0,1  \}$$
but we homogenize it to obtain the equation

$$ z_1^n = y_0 ^{m_0} y_1^{m_1}  (y_1 - y_0 )^{m_2}   (x_0 y_1 - x_1 y_0 )^{m_3}     x_0^{n-m_3}. \  $$

The above equation describes a singular surface $\Sigma'$ which is a cyclic covering of 
 $\PP^1 \times \PP^1$ with group $ G : = \ZZ / n$; $\Sigma'$ 
is contained inside the line bundle $\LL_1$  
over $\PP^1 \times \PP^1$  whose sheaf of holomorphic sections $\sL_1$ equals $ \hol_{\PP^1 \times \PP^1}(1,1)$.

The first  projection $\PP^1 \times \PP^1 \ra  \PP^1$ induces a morphism $p : \Sigma' \ra  \PP^1$
and we consider the curve $B$, normalization of the covering of $\PP^1$ given by
$$ w_1^n =  x_0 ^{n_0} x_1^{n_1}  (x_1 - x_0 )^{n_2}  .$$ 

We consider the normalization $\Sigma$ of the  fibre product $\Sigma' \times_{\PP^1} B$. 

$\Sigma$ is an abelian covering of $\PP^1 \times \PP^1$ with group $(\ZZ / n)^2$, and the local monodromies are as follows:

$$ \{ y = \infty \} = \{ y_0=0\} \mapsto (m_0,0) , \ \{ y = 0 \} = \{ y_1=0 \} \mapsto (m_1,0), \ $$
$$\{ y = 1 \} = \{ y_1=y_0 \} \mapsto (m_2,0),$$ 
$$ \{ x = \infty \} \mapsto (n- m_3,n_0) , \ \{ x = 0 \} \mapsto (0, n_1), \ \{ x = 1 \} \mapsto (0, n_2),$$
$$\Delta := \{  (x_0 y_1 - x_1 y_0 ) = 0\}  \mapsto (m_3,0).$$

Since the branch divisor is a not a normal crossing divisor, we blow up the three points  $ P_0 : = \{ x_0 = y_0 = 0\}$,
 $ P_1 : = \{ x_1 = y_1 = 0\}$,  $ P_2 = \{ x_0 -x_1 = y_0 - y_1= 0\}$.
 
We obtain in this way a del Pezzo surface which we denote by $Z$, we denote by $E_i$ the exceptional $(-1)$-curve 
inverse image of the point $P_i$, and we notice that the pull back of the branch divisor is now  a normal crossing divisor.

The local monodromies around the three exceptional divisors are now

$$ E_0 \mapsto (m_0, n_0), \  E_1 \mapsto (m_1 + m_3, n_1), E_2 \mapsto (m_2 + m_3, n_2).$$

We finally define $S$ to be the normalization of the pull back  $\Sigma \times_{  \PP^1 \times \PP^1}Z$.

\begin{prop}\label{p1}
The surface $S$ is smooth, and each irreducible component of the branch locus $\sB$ has local monodromy of order $n$.

\end{prop}

\begin{proof}
Given two irreducible components $\sB_i$, $\sB_j$, they are smooth and they intersect transversally in exactly one
point , or are  disjoint. Hence it is sufficient to show 

\begin{itemize}
\item
that the inertia subgroups (image of the local monodromy) are cyclic of order $n$
\item
if $\sB_i$, $\sB_j$ intersect, the corresponding inertia subgroups generate $(\ZZ / n)^2$.

\end{itemize}

By our assumptions all the local monodromies are elements of order $n$, hence the first assertion.

The second assertion follows from  the following fact: $(\ZZ / n)^2$ is  generated by pairs of the form

$$ (a, 0) , (0,b), a,b \in (\ZZ / n)^*,$$ 
$$ (a, 0) , (c,b), a,b \in (\ZZ / n)^*,$$
$$ (0, a) , (b,c), a,b \in (\ZZ / n)^*,$$

or of the  form  
$$ (n-m_3, n_0) , (m_0, n_0),$$
since their span is the span of $ (n-m_3-m_0, 0) , (m_0, n_0),$
and $n_0, m_3 + m_0$ are units in $(\ZZ / n)$.

\end{proof}

\begin{prop}\label{p2}
Let $f : S \ra B$ be the morphism induced by the projection $\Sigma \ra B$. Then

\begin{enumerate}
\item
the genus of the fibres $F$ equals $g = n-1$
\item
the genus of the base curve $B$ equals $b = \frac{n-1}{2}$
\item
all the fibres are smooth, except the fibres over $x=0, x=1, x=\infty$, which consists of two smooth curves
of genus $b$ intersecting transversally in exactly one point.
\item
$f$ is the Albanese map of $S$, i.e., $b = q : = h^1 (\hol_S) = h^1 (\hol_S (K_S))$.

\end{enumerate}

\end{prop}

\begin{proof}
(1) and (2) follow from Hurwitz' formulae:
$$ 2(g-1) = -2n + 4 (n-1), \  2(b-1) = -2n + 3 (n-1),$$
since the general fibre is a $(\ZZ / n)$-cyclic cover of $\PP^1$ totally ramified in 4 points,
while $B$ is a $(\ZZ / n)$-cyclic cover of $\PP^1$ totally ramified in 3 points.

(3): the fibres over $x=0, x=1, x=\infty$ are the inverse images of two smooth curves meeting transversally in exactly one point $P$
and which are part of the branch locus $\sB$. The covering is totally branched on $P$, hence these special fibres consist of two
smooth curves  meeting transversally in exactly one point $P'$. Both are $(\ZZ / n)$-cyclic covers of $\PP^1$ totally ramified in three points,
hence their genus equals $b$.

The other fibres are the inverse image of a $\PP^1$  intersecting  the branch locus transversally in 4 points, hence they are all smooth of genus $g=n-1$.

(4) There are several ways to prove that $b=q$, some more explicit,  the following one is in the spirit of this paper.

We want to calculate 
$$q : = h^1 (\hol_S) = h^1 (\hol_S (K_S))=  h^1 (\omega_{S|B} (K_B))$$ 
and we denote $\omega_{S|B} $  for short by $\omega$.

We use the Leray spectral sequence for $f$, saying that 
$$ q = h^0(B, \sR^1 f_* (\omega)(K_B) ) + h^1 ( f_* (\omega)(K_B) ) .$$

The first term, by relative duality, equals $b = h^0(B, \hol_B(K_B) )$, while the second vanishes, as $ f_* (\omega) = V = A \oplus Q$
by Fujita's second theorem. Then $h^1 ( V(K_B) )= h^0 (V^{\vee})$ by Serre duality, and $ h^0 (A^{\vee} )= 0$
since $A$ is ample, while   $ h^0 (Q^{\vee} )= 0$, else the monodromy of some summand of the unitary flat bundle $Q$
would have trivial monodromy, contradicting the irreducibility of the monodromy representation.

\end{proof}

\begin{prop}\label{p3}
The smooth surface $S$ is minimal of general type with $K_S$ ample, and with invariants
$$ e(S) = c_2(S) =  3 + 2(n-2)(n-3) = 2 n^2 - 10 n + 15;  $$
$$ K_S^2 =  5 (n-2)^2 . $$

They have positive index $\sigma(S)= \frac{1}{3} (K^2_S - 2 e(S)) > 0$  and indeed their slope $\frac{K^2_S}{  e(S)}   \geq 2,5$.

We have that the universal cover of $S$ is the unit ball in $\CC^2$ if and only if $n=5$, which corresponds to the case of
  three distinct surfaces $S', S'', S'''$.

\end{prop}

\begin{proof}

The calculation for the topological Euler-Poincar\'e characteristic $e(S)$ follows from the Zeuthen Segre formula
asserting that $e(S)$ equals the sum of the product $ e(B) e(F) = 4 (b-1)(g-1)$ with the number $\mu$ of singular 
fibres counted with multiplicity: here therefore $\mu = 3$ and we get
$ e(S) =  3 + 2(n-2)(n-3) = 2 n^2 - 10 n + 15  $.

To calculate $K_S^2$, we observe that $K_S$ is numerically equivalent to the pull back of $K_Z + \frac{n-1}{n} \sB$.

Since $\sB \equiv 4 L_1 + 4 L_2 - 2 \Sigma_i E_i$, where $L_1, L_2$ are the total transforms of the two rulings of 
$\PP^1 \times \PP^1$,  and $ K_Z =  -2  L_1 -2  L_2 + \Sigma_i E_i$, we obtain that $\sB \equiv - 2 K_Z$,
hence $K_S$ is numerically equivalent to the pull back of $K_Z + \frac{n-1}{n} \sB =  -  \frac{n-2}{n} K_Z$.

Since $-K_Z$ is ample, and $ K_Z^2 = 5$, we easily obtain that $S$ has has ample canonical divisor $K_S$, and 

$$K_S^2 =   5 (n-2)^2 .$$

The surface $S$ is minimal since $K_S$ is ample.

We now calculate the slope as 
$$\frac{5 (n-2)^2}{ 3 + 2(n-2)(n-3)} = \frac{5}{2} \frac{n-2}{n-3 +  \frac{3}{2(n-2)}}  > \frac{5}{2}.$$

The same formula shows that the slope is a strictly decreasing function of $n$, tending to $\frac{5}{2}$ as $ n \ra \infty$,
and beginning with slope $= 3$ for $n=5$. But, by the theorem of Yau, slope equal to $3$ is equivalent to having 
the ball as universal cover.

Consider now the case $n=5$: the 4-tuple of residue classes modulo $5$ is equivalent, modulo simultaneous multiplication by a unit,
to $ 1 + 1 + 1 + 2 = 5$, and this is the only representation via integer rests which add up to $5$. Also the $n_i$ are
uniquely determined as $ 1 + 2 + 2 = 5$.

There are two different cases: $m_3 \neq m_i$, or (up to renumbering) $m_3=m_0= m_1$; in this second case  there are 
two subcases, according to $n_0 = n_1 $ or $n_0 \neq n_1$.

\end{proof}

\begin{rem}
The above three surfaces which occur for $n=5$ have already been constructed in \cite{bc}.

Being ball quotients, they are rigid.

 In joint work of the first author together with Ingrid Bauer it was recently shown that the above surfaces $S$ for $ n \geq 5$ are rigid.

Another interesting question is whether the surfaces $S$ are always $ K (\pi,1)$'s, i.e., whether their
universal covering is always contractible.

\end{rem}

 Recall now  the following algebraic formula for the Euler number, the so-called Zeuthen-Segre formula (see the  lecture notes \cite{cb}).

\begin{definition}
Let $ f : S \ra B$ be a fibration of a smooth algebraic surface $S$ onto a curve of genus $b$, and consider a fibre 
$F_t  = \sum n_i C_i$, where the $C_i$ are irreducible curves.

Then the divisorial singular locus of the fibre is defined as the divisorial part of the critical scheme,
$ D_t : =  \sum (n_i -1) C_i$, and the Segre number of the fibre is
defined as 
$$\mu_t : = deg \sF + D_t K_S  - D_t^2,$$
where the sheaf  $\sF$  is concentrated in the singular points of the reduction of the
fibre, and is the quotient of $\hol_S$ by  the ideal sheaf  generated by the components of the vector 
$d \tau / s$, where $ s = 0 $ is the equation of $D_t$, and where $\tau$ is the pull-back of a 
local parameter
at the point $t \in B$.

More concretely, $$ \tau  = \Pi_j   f_j^{n_j} , s = \tau / ( \Pi_j   f_j),$$ and the logarithmic derivative yields

$$ d \tau = s [  \sum_j n_j (df_j \Pi_{h \neq j} f_h].$$

\end{definition} 

The following is the refined Zeuthen-Segre formula
\begin{theo}
Let $ f : S \ra B$ be a fibration of a smooth algebraic surface $S$ onto a curve of genus $b$, and with fibres of genus $g$.

Then $$c_2(S) = 4 (g-1)(b-1) + \mu,$$ where $\mu = \sum_{t\in B} \mu_t$,
and $\mu_t \geq 0$ is defined as above. Moreover, $\mu_t$ is strictly positive, except if the fibre is smooth or a multiple of a smooth
curve of genus $g=1$. 

\end{theo}

\begin{prop}
Let $S$ be one of the surfaces considered in this section. Then any surface $X$ which is homeomorphic to $S$
has Albanese map which is a fibration onto a curve $B$ of the same genus $b= 1/2(n-1)$ as the Albanese image of $S$. If moreover $X$
is diffeomorphic to $S$,  the Albanese fibres have the same genus $g = 2b= n-1$
and, if the number of singular points on the fibres is finite, there are only three singularity on the fibres, counted with multiplicity.
In particular, there are at most three singular fibres.
\end{prop}

\begin{proof}
The first two statements follow directly from \cite{isogenous}, theorem A.

For the last statement we invoke the above refined Zeuthen-Segre formula
$$ e(X) =  4(b-1)(g-1) + \mu.$$ 

 Since $b,g$ are the same for $S$ and $X$, it follows that $\mu=3$, which shows the third assertion.

The refined version of the Zeuthen-Segre formula implies in particular  that, if $D$ is the divisorial part of the critical locus,
then $3= \mu \geq D K_X - D^2$, where $D K_X - D^2 = 2 p(D) -2 - 2 D^2$ is a positive even number.

Each non reduced  fibre $F_t = \Sigma_i n_i C_i$ gives a contribution $D_t : =  \Sigma_i (n_i -1) C_i$ to $D$, and   Zariski's lemma 
says that, if $D_t \neq 0$, then $D_t^2 < 0$ unless $F_t$ is a multiple fibre. 

If we had a multiple fibre $F_t = mC$, then we would have $D_t K_X - D_t^2 = D_t K_X = (m-1)/m F K_X=  (2g-2)(m-1)/m \geq (g-1) \geq 4$,
which is a contradiction. Hence there are no multiple fibres.

Assume that $F_t$ is a non reduced fibre, so that $D_t \neq 0$: then  $D_t K_X - D_t^2 = 2$, since it is a strictly positive  even integer which is not greater than $3$.

So,  if there are infinitely many singular points on the fibres, then there is exactly one non reduced fibre and at most one more singular point;
in particular, there 
are at most two singular fibres.

\end{proof}

We can summarize our main result in the following theorem,
for which we give two proofs, one self-contained and based on Fujita's second theorem, the other based on the theory
of hypergeometric integrals.

  \begin{theo}\label{surfaces}
  There exists an infinite series of surfaces with ample canonical bundle,  
  whose Albanese map is a  fibration $ f : S \ra B$  onto a curve $B$ of genus $b= 1/2(n-1)$,  with fibres of genus $g = 2b= n-1$; 
  here $n \geq 5$ can be any integer relatively prime with $6$ and  $f$ is as in Prop.~\ref{p2}.
  
  These Albanese  fibrations  yield negative answers to Fujita's question about the semiampleness of $V : = f_* \om_{S|B}$, since
  here $V : = f_* \om_{S|B}$
splits as a direct sum 
$ V = A  \oplus Q$, where $A$ is an ample    vector bundle, and $Q$ is a  unitary  flat  bundle with
infinite monodromy group.

The fibration $f$ is semistable: indeed all the fibres are smooth, with the exception of three fibres which are the union of two smooth curves of genus $b$
which meet transversally in one point.

For $n=5$ we get   three surfaces which are rigid, and are quotient of the unit ball in $\CC^2$ by a torsion free cocompact lattice $\Ga$.

The rank of $A$, respectively $Q$ is in this case equal to $2$.
 \end{theo}
 
 \begin{proof} 
 By Propositions \ref{p1},  \ref{p2}, \ref{p3}, the only assertion which needs to be shown is that, if we considering the splitting of
   $V : = f_* \om_{S|B}$
as a direct sum 
$ V = A  \oplus Q$, where $A$ is an ample    vector bundle, and $Q$ is a  unitary  flat  bundle, then  $Q$ has 
infinite monodromy group.

We observe that the group $G = \ZZ/n$ acts on the fibration, thus we have a splitting according to the characters of $G$,
$ j \in \ZZ/n$,  $$V = \oplus_{ j  \in \ZZ/n}  V_j $$.

The fact that all the fibres are smooth and that the only singular fibres  are two smooth curves intersecting transversally in one point
shows that the vanishing cycles are homologically trivial. Hence the local monodromies in cohomology are trivial, thus we have 
a flat vector bundle $\sH : =   R^1f_* (\CC)$, which is a holomorphic flat bundle having $V$ as a holomorphic subbundle.

Similarly we have a splitting 

$$ \sH = \oplus_{ j  \in \ZZ/n}  \sH_j ,$$
where the flat bundles $\sH_j$ have all rank $2$ for $j\neq 0$, as observed in section three. 

Moreover, we have the following direct sum of complex vector bundles   $\sH_j = V_j   \oplus \overline{ V_{-j}}$,
and there are a priori several  possible  cases:

\begin{itemize}
\item
 $\sH_j = V_j  $ and $V_{-j} = 0$, hence $V_j$ is a flat holomorphic bundle; in this case the bundle $\sH_j $ carries a flat Hermitian form which is positive definite;
 \item
  $\sH_j = \overline{ V_{-j}}$ and $V_j = 0$, hence $V_{-j}$ is a flat holomorphic bundle; in this case the bundle $\sH_j $ carries a flat Hermitian form which is negative definite;
  \item
$\sH_j = V_j   \oplus \overline{ V_{-j}}$, both summands have rank 1, and here  the bundle $\sH_j $ carries a flat Hermitian form which is indefinite.
This case could a priori bifurcate in the cases $V_j  $ is flat, or $V_j $ is ample (i.e., it has strictly positive degree). 

\end{itemize}

{\bf First Proof:}

Step 1: $V$ is not flat.

In fact, otherwise  (see for instance  theorem 4 of \cite{cd}) 
$$ 0 = 12 \ deg (V) = K_S^2 - 8 (g-1)(b-1) ; $$
hence $ K_S^2 = 8 (g-1)(b-1) = 2 e(S) - 6$, contradicting Proposition \ref{p3}.

Step 2: hence $V = \oplus_j V_j $ admits an ample rank 1 summand $V_j = A_j$.

Step 3: It suffices to prove the theorem in  the case where $n$ is prime.

In fact, if $k$ divides $n$, $ n = hk$, we have an anologous fibration  $f_k : S_k \ra B_k$ for the surface $S_k$ obtained by taking the associated $(\ZZ/k)^2$ covering.

Pulling back the fibration to $B$, under $ \psi : B \ra B_k$, we obtain a surface $S'  = S / G'$, where $G' = \ZZ/ h$; and the fibration $f$ factors through
$ f' : S' \ra B$. Hence $V'  =  \psi^* (V_k)$ is a direct summand of $V$ and we are done, since $V'$ has a unitary flat summand  $Q'$ with infinite monodromy.

Step 4: There is an eigenbundle $\sH_j$ with infinite monodromy.

This follows from Step 2 and the following lemma.
\begin{lemma}
If $V_j = A_j$ is an ample rank 1 summand, then $\sH_j$ is irreducible and with infinite monodromy.

\end{lemma}

\begin{proof}{\bf (of the Lemma).}
We first show that the rank two flat vector bundle is irreducible. Otherwise there would be an exact sequence of flat vector bundles
$$0 \ra \sH' \ra \sH_j \ra \sH'' \ra 0 $$
where both $\sH' , \  \sH''  $ have rank 1.

Since $\sH_j = V_j   \oplus \overline{ V_{-j}}$, we get a nontrivial homomorphism $V_j \ra \sH_j$
which realizes $V_j$ as a holomorphic subbundle. Composing with the above surjection $\sH_j \ra \sH'' $
we get a holomorphic homomorphism $V_j \ra \sH''$, which must be zero since the target has degree zero, while
$V_j = A_j $ has positive degree. We deduce a nontrivial holomorphic homomorphism $V_j \ra \sH'$, which must be zero
by the same argument, and we have found a contradiction to the fact that $V_j \ra \sH_j$ is injective.
\end{proof}

Step 4. Observe preliminarily that our surfaces, the Albanese map $f$ and  and all the bundles $V$, $Q$, are defined over $\ZZ$.

Now,  by construction we have  a flat rank $2$ summand $V_{-1} = \sH_{-1}$  (since $m_0 + m_1 + m_2 + m_3 = n$). Hence, when  $n$ is prime,
$\sH_{-1}$ and $\sH_j$ are Galois conjugate. The condition that the monodromy is infinite is obviously  invariant under Galois conjugation
(since a finite group of matrices  transforms to a finite group under a field automorphism).

Hence also $V_{-1}  = \sH_{-1} $  has infinite monodromy, and it is a direct summand of $Q$ with infinite monodromy.

\smallskip

{\bf Second Proof:}

By corollary \ref{corinfinite} there is $j \in (\ZZ/n)^*$ such that $\hat{\HH}_j$ carries a monodromy invariant indefinite Hermitian form $H_j$, 
and is irreducible with infinite monodromy.

Therefore also $\sH_j$ has infinite monodromy. Since $j$ is a unit, it follows that $\sH_{-1}$ and $\sH_j$ are Galois conjugate. 
Hence $V_{-1}  = \sH_{-1} $  has also  infinite monodromy, and the same holds for   $Q$, of which  $V_{-1} $ is a direct summand. 
 \end{proof}
 
 \begin{rem}
 
 In the standard case $V$  has a lot of flat summands.
 
  In fact,  $V_j = 0 $ for $j \leq \frac{n}{3}$ (since $3 j \leq n$ implies $ j + j + j + [j(n-3)] < 2n \Rightarrow
 j + j + j + [j(n-3)] = n$);
 hence $V_{-j}  $  i flat for $j \leq \frac{n}{3}$.
 
 On the other hand, for n=11 we can take $m_0 = 1, m_1 = 2, m_2 = 3, m_3 = 5$ and then only $V_{10}$ is a flat summand,
 of course with infinite monodromy.
 \end{rem}

\section{General observations and relation with Shimura curves}

Consider our surfaces $ S \ra B$ as yielding a curve inside the compactified moduli space of curves of genus $ g = n-1$. 
The image of $B$ inside $\overline{\frak M_g}$ intersects the boundary only in points belonging to the divisor $\Delta_{g/2,g/2}$.

Moreover, under the Torelli map $\frak M_g \ra \frak A_g$, the image does not go to the boundary, since the singular fibres have
compact Jacobian.

For $n=5$ we obtain a rigid curve inside  $\overline{\frak M_g}$, a phenomenon which is not new: compare the examples provided by double Kodaira fibrations
(\cite{cat-rollenske}).

Now, $B$ parametrizes all the curves with an action of $\ZZ/n$ whose quotient is $\PP^1$, and with branch locus $\sS$  consisting of 4 points: because
all deformations preserving the symmetry come from $H^1(C, \Theta_C)^G$ which is isomorphic to  $H^1(\PP^1, \Theta(- \sS))$, the space of
logarithmic deformations of the pair consisting of $\PP^1$ and the 4 points on it. 

$B$ parametrizes, via the Torelli
map, also principally polarized Abelian surfaces with such a symmetry. 

The question is whether the symmetry-preserving deformations of these Abelian varieties are just the ones parametrized by $B$.

 The main point is that (see \cite{pavia} and \cite{pavia2}, especially for more details concerning  the relation with Shimura curves) the dual of $H^1(C, \Theta_C)^G$ equals $H^0(2 K_C)^G$, while the tangent space to the 
symmetry preserving deformations of the Abelian varieties  
is given by 
$$ Sym^2 ( H^0(K_C))^G =  Sym^2 ( \bigoplus_j V_j)^G = \bigoplus_{j \leq n/2} (V_j \otimes V_{-j} ).$$

Observe that $V_0 = 0$, while, for a character $j$, writing as usual  $\mu_{i,j} = \frac{1}{n}[m_i j]$,  the condition $\sum_i  \mu_{i,j} = 2$ is equivalent
to dim ($V_j \otimes V_{-j}$ )=1, else one has dim ($V_j \otimes V_{-j}$ )=0.

In other words, the number of parameters for the symmetry-preserving deformations of these Abelian varieties is just the number of rank 2
ample bundles in the direct image sheaf $f_* (\omega)$. 

If there is only one such ample summand, then this means that we have a Shimura curve in $\frak A_g$.
This situation leads to a finite number of cases, which were classified  by Moonen in \cite{moonen} (see \cite{pavia2} for groups more general than cyclic groups).

Interest in these Shimura curves is due to a conjecture by Oort that there should not be such curves as soon as $g$ is bigger than 7,
see \cite{lu-zuo} and references therein for results in this direction.

\section{Appendix}
\begin{prop}\label{propapp} Let $m_0,m_1,m_2,m_3, n\in \ZZ$ with $ 0 < m_k \leq n-3\, (0\leq k\leq 3)$ and $m_0 + m_1 + m_2 + m_3= n.$ 
For $j\in 1,\ldots,n-1,$ let $\hat{\HH}_j$ be the local system as in Section~\ref{sectioncyclic}. 
Assume additionally
that each of the numbers $m_0,\ldots,m_3$ is coprime to $n$ (resp., assume that  that $j$ is coprime to $n$ with no further assumption on $m_0,\ldots, m_3$).   
Then the local systems  $\hat{\HH}_j$ are irreducible for $j=1,\ldots,n-1$ (resp., for $j$ prime to $n$).
\end{prop}

\begin{proof}  By construction, the local sections of $\hat{\HH}_{-j}$ are variations in $x$ of periods on  the desingularizations of the curves 
$$z_1^n = y_0^{m_0} y_1^{m_1} (y_1 - y_0 )^{m_2}  (y_1 - x y_0 )^{m_3} , \ \  x \in \CC \setminus \{0,1  \}$$
of the form (given on the the affine part belonging to $y_1=1$) 
$$ \int_\gamma \frac{y_0^s(1-y_0)^t(1-xy_0)^u }{z_1^j}dy_0,$$  where $s,t,u$ are integers, cf.~\cite{Wolfart}, Section~2.
It is convenient to introduce integers $A,B,C$ and rational numbers $a,b,c$ by the following conditions:
$$ A=(1-b)n=m_0,\, B=(b+1-c)n=m_2,\, C=an=m_3,\,  n-A-B-C=m_1.$$
Therefore
$$ a=\frac{m_3}{n},\quad b=1-\frac{m_0}{n},\quad c=2-\frac{m_0}{n}-\frac{m_2}{n}.$$
If $\gamma$ denotes integration from $0$ to $1$ then  the above integral can be expressed as a hypergeometric function as follows:
$$ \int_0^1 \frac{y_0^s(1-y_0)^t(1-xy_0)^u }{z_1^j}dy_0 =D\cdot {}_2F_1(ja-u,jb-j+1+s,jc-2j+2+t+s;x)$$
where $D$ is a constant in $\CC$,
cf.~\cite{Wolfart}, Formula~(7). Hence, in order to show that $\hat{\HH}_j$ is irreducible, it suffices to show that 
the hypergeometric differential 
equation belonging to ${}_2F_1(ja-u,jb-j+1+s,jc-2j+2+t+s;x)$ is irreducible. This is the case if and only if 
the values $ja-u,\,jb-j+1+s$ and the differences $(jc-2j+2+t+s)-(ja-u), \,(jc-2j+2+t+s)-(jb-j+1+s)$ are not contained in $\ZZ,$
cf.~\cite{Beukers_Lectures}, Cor.~3.10. Obviously, the latter 
condition holds if and only if  the values 
$ja=\frac{jm_3}{n},\,jb=-\frac{jm_0}{n}+j$
and 
$$ ja-jc=\frac{jm_3}{n}+\frac{jm_0}{n}+\frac{jm_2}{n}-2j=-\frac{jm_1}{n}-j$$ as well as 
$$ jb-jc=j-\frac{jm_0}{n}+\frac{jm_0}{n}+\frac{jm_2}{n}-2j=\frac{jm_2}{n}-j$$ are not contained in $\ZZ.$ This holds by 
our assumptions.
\end{proof}

\bigskip
\section*{Acknowledgments}
We thank Osamu Fujino for asking us to produce  simple explicit semistable fibrations which are  counterexamples to Fujita's question.
We thank also  Stefan Reiter for his help for the proof of lemma \ref{lemmafinite}, Paola Frediani and Elisabetta Colombo for a useful conversation on Shimura curves
and 
Alessandro Ghigi for valuable comments which pushed us  to give more details in  the proof.


\begin{thebibliography}{0}
\bibitem{atiyah}
M. Atiyah, { Complex analytic connections in fibre bundles } 
Trans. Amer. Math. Soc. 85 (1957),  181--207.

\bibitem{Barja}
M. A. Barja, { On a conjecture of Fujita}, preprint UPC,  Barcelona (1998), pp. 1--10.

 \bibitem{bc}
 I. C. Bauer, F. Catanese
{ A volume maximizing canonical surface in 3-space. }
Comment. Math. Helv. {\bf 83}, No. 1 (2008),  387--406.

\bibitem{BH}
F. Beukers and G. Heckman, { Monodromy for the hypergeometric function ${}_nF_{n-1}$,}
Invent. Math.  95 (1989), 325--354.

\bibitem{Beukers_Lectures} F. Beukers, Gauss' hypergeometric function, in {\em Arithmetic and Geometry Around Hypergeometric Functions}, Progress in Math., 260 (2007), 23--42. 

 \bibitem{cb}
F. Catanese,  I. C. Bauer, 
{ ETH Lectures on algebraic surfaces},
preliminary version (2004).
 
\bibitem{isogenous}
F. Catanese,  {  Fibred surfaces, varieties isogenous to a product and related moduli spaces},
 Amer. J. Math. 122 (2000), no. 1, 1--44.

\bibitem{cat-rollenske}
F. Catanese, S. Rollenske, 
{ Double Kodaira fibrations.} 
 J. Reine Angew. Math. 628 (2009), 205--233. 

\bibitem{cd}
F. Catanese, M.   Dettweiler, 
{ Answer to a question by Fujita on Variation of Hodge Structures}, 
arXiv:1311.3232, 26 pages, to appear in a volume of `Advanced Studies in Pure Mathematics'
 dedicated to Yujiro Kawamata on the occasion of his 60-th birthday.

\bibitem{cd2}
F. Catanese, M.   Dettweiler, 
{ The direct image of the relative dualizing sheaf needs not be semiample,} C. R. Math. Acad. Sci. Paris 352 (2014), no. 3, 241--244.

\bibitem{pavia}
E. Colombo, P. Frediani, A. Ghigi, 
{ On totally geodesic submanifods in the Jacobian locus},
arXiv 1309.1022v2.

\bibitem{CurtisReiner}
C.W. Curtis,  I. Reiner, { Representation Theory of Finite Groups and Associative Algebras,}
 Wiley (1962). 

\bibitem{del}
P. Deligne,
{ \'Equations diff\'erentielles \'a points singuliers r\'eguliers,}
 Lecture
Notes in Mathematics, Vol. 163. Springer-Verlag, Berlin-New 
York (1970), pp.  iii+133.


\bibitem{DeligneMostow}
P.~Deligne and G.D. Mostow.
\newblock Monodromy of hypergeometric functions and non-lattice integral
  monodromy.
\newblock { Publ. Math. IHES}, 63:\ (1986), 5--89.

          \bibitem{drk}
M. Dettweiler, S. Reiter, { Rigid local systems and motives of type G2, with an Appendix
by M. Dettweiler and N. Katz}, Compositio Mathematica 146 (2010), 929--956.

     \bibitem{ds}
M. Dettweiler, C. Sabbah, { Hodge theory of the middle convolution, } Publ. Math. RIMS Kyoto 49 (2013), 761--800.

\bibitem{e-v} H. Esnault, E. Viehweg, 
{ Effective bounds for semipositive sheaves and for the height of points on curves over complex function fields. }
Algebraic geometry (Berlin, 1988). 
Compositio Math. 76, no. 1--2 (1990), 69--85. 

\bibitem{pavia2}
P. Frediani, A. Ghigi, M. Penegini,
{ Shimura varieties in the Torelli locus via Galois coverings},
arXiv 1402.0973v3.

\bibitem{ffs14}
O. Fujino, T. Fujisawa, M. Saito, 
{ Some remarks on the semipositivity theorems.}
 Publ. Res. Inst. Math. Sci. 50, no. 1  (2014), 85--112. 
 
 \bibitem{ff14}
O. Fujino, T. Fujisawa,
 { Variations of mixed Hodge structure and semipositivity theorems. }
 Publ. Res. Inst. Math. Sci. 50 (2014), no. 4, 589--661.
     
  \bibitem{fuj1}
  T.  Fujita, {  On K\"ahler fiber spaces over curves. } J. Math. Soc. Japan 30 (1978), no. 4, 779--794.

   \bibitem{fuj2}
 T.  Fujita,  { The sheaf of relative canonical forms of a K\"ahler fiber space over a curve. } Proc. Japan Acad. Ser. A Math. Sci. 54 (1978), no. 7, 183--184. 

\bibitem{griff1}
P. Griffiths, { Periods of integrals on algebraic manifolds.I. Construction and properties of the modular varieties. -II. Local study of the period mapping.}  Amer. J. Math. 90 (1968) 568--626 and 805--865.

\bibitem{griff2}
P. Griffiths, {  Periods of integrals on algebraic manifolds. III. Some global differential-geometric properties of the period mapping.} Inst. Hautes \'Etudes Sci. Publ. Math. No. 38 (1970) 125--180.

 \bibitem{G-S}
 P. Griffiths, W. Schmid, 
{ Recent developments in Hodge theory: A discussion of techniques and results. }
Discrete Subgroups of Lie Groups Appl. Moduli, Pap. Bombay Colloq. 1973 (1975), 31--127.

\bibitem{Griffiths_Topics}
P.~Griffiths.
\newblock { Topics in transcendental algebraic geometry}.
\newblock Number 106 in Annals of Mathematics Studies. Princeton University
  Press, (1984).
  
   
  \bibitem{haraoka}
   Y. Haraoka,  {  Finite monodromy of Pochhammer equation}, 
    Ann. Inst. Fourier (Grenoble) 44 (1994), no. 3, 767--810. 
    
 

\bibitem{hartVBC}
R. Hartshorne, { Ample vector bundles on curves}, Nagoya Math. J.,  43 (1971), 73--89.

\bibitem{katata}
{ Open problems: Classification of algebraic and analytic manifolds.}
Classification of algebraic and analytic manifolds, Proc. Symp.
Katata/Jap. 1982.
   Edited by K. Ueno.
Progress in Mathematics, 39.
Birkh\"auser, Boston, Mass. (1983), 591--630.


\bibitem{Ka96}
N. Katz, { Rigid local systems.} Annals of Mathematics Studies 139, Princeton Press (1996).


\bibitem{kaw0}
Y. Kawamata,  { Characterization of abelian varieties.} Compositio Math. 43 (1981), no. 2, 253--276.

\bibitem{kaw1}
Y. Kawamata,   { Kodaira dimension of algebraic fiber spaces over curves.}
  Invent. Math. 66 (1982), no. 1, 57--71.

\bibitem{kaw2}
Y. Kawamata,  { On algebraic fiber spaces. }
Contemporary trends in algebraic geometry and algebraic topology (Tianjin, 2000), Nankai Tracts Math., 5, 
World Sci. Publ., River Edge, NJ, (2002), 135--154.
  
  \bibitem{toroidalEmb}
  G. Kempf, F.F. Knudsen, D. Mumford, B. Saint Donat,
  { Toroidal embeddings, I} Springer Lecture Notes in Mathematics, 739 (1973), viii + 209 pp. .
 
   \bibitem{kod}
  K. Kodaira, 
  { A certain type of irregular algebraic surfaces,}
    J. Anal. Math. 19 (1967), 207--215.


 \bibitem{kollar}
 J. Koll\'ar, 
 {  Higher direct images of dualizing sheaves. I, II}
Ann. Math. (2) 123 (1986), 11-- 42 ,124 (1986),  171--202.
 

 \bibitem{kollar2}
 J. Koll\'ar, 
 { Subadditivity of the Kodaira dimension: Fibers of general type,}
Algebraic geometry, Proc. Symp., Sendai/Jap. 1985, Adv. Stud. Pure Math. 10 (1987), 361--398.

\bibitem{lu-zuo}
X. Lu, K.  Zuo, 
{ The Oort conjecture on Shimura curves in the Torelli locus of curves},
arXiv 1405.4751v2.

  \bibitem{moonen}
  B. Moonen, {  Special subvarieties arising from families of cyclic covers of the projective line},
  Doc. Math. 15 (2010), 793--819.
  
  \bibitem{NS}
   M. S. Narasimhan,  C. S. Seshadri,
   { Stable and unitary vector bundles on a compact Riemann surface.}
    Ann. of Math. (2) 82 , (1965),  540--567. 

 
 \bibitem{Peters}
 C. A. M. Peters,  { A criterion for flatness of Hodge bundles over curves and geometric applications. } Math. Ann. 268 (1984), no. 1, 1--19.


 \bibitem{schmid}
 W. Schmid, 
{ Variation of Hodge structure: The singularities of the period mapping. }
Invent. Math. 22 (1973), 211--319. 

\bibitem{Schwarz}
H.A. Schwarz.
{ \"Uber diejenigen F\"alle in welchen die Gaussische
  hypergeometrische Reihe eine algebraische Funktion ihres vierten Elements
  darstellt}.
 Journal Reine u. Angew. Math. , 75 (1873), 292--335.

\bibitem{um}
H. Umemura, 
{  Some results in the theory of vector bundles. } Nagoya Math. J. 52 (1973), 97--128.

\bibitem{weil}
A. Weil, 
{  G\'en\'eralisation des functions ab\'eliennes, }J. Math. Pures Appl. (9) 17 (1938), 47--87.

\bibitem{Wolfart} J. Wolfart, Werte hypergeometrischer Funktionen, Invent. Math. 92 (1988), 187--216.

\bibitem{zucker}
S. Zucker, 
{ Hodge theory with degenerating coefficients: $L^2-$ cohomology in the Poincar\'e metric. }
Ann. Math. (2) 109 (1979), 415--476.

\bibitem{zucker-remarks}
S. Zucker, 
{ Remarks on a theorem of Fujita. }
J. Math. Soc. Japan 34 (1982), 47-54.

\bibitem{zuckerAMS}
S. Zucker, 
{ Degeneration of Hodge bundles. (After Steenbrink). }
Topics in transcendental algebraic geometry, Ann. Math. Stud. 106 (1984), 121--141.

\end{thebibliography}
\end{document}